\documentclass{article}

\usepackage[latin1]{inputenc}
\usepackage[T1]{fontenc}

\usepackage{bbm}
\usepackage{xypic}

\usepackage{amsmath}
\usepackage{amsfonts}
\usepackage{amssymb}
\usepackage{amsthm}
\usepackage{graphicx}
\usepackage{enumerate}

\newtheorem{theorem}{Theorem}

\newtheorem{definition}[theorem]{Definition}
\newtheorem{lemma}[theorem]{Lemma}

\newtheorem{proposition}[theorem]{Proposition}
\newtheorem{remark}[theorem]{Remark}
\numberwithin{equation}{section}

\newcommand{\PP}{\mathbb{P}}
\newcommand{\PPP}{\hat{\mathbb{P}}}

\newcommand{\EE}{\mathbb{E}}

\newcommand{\FF}{\mathcal{F}}
\newcommand{\ff}{f}
\newcommand{\sig}{\sigma_N^{\epsilon}}
\newcommand{\Sca}{\mathcal{S}_N}

\begin{document}
\title{Inhomogeneity and universality: off-critical behavior of interfaces}
\author{Pierre Nolin}
\date{Courant Institute of Mathematical Sciences}
\maketitle


\begin{abstract}
We further study the interfaces arising in a situation of inhomogeneity. More precisely, we identify a characteristic length for the gradient percolation model, that enables us to tighten previous estimates established for it. This allows to construct non-trivial scaling limits: the limiting objects share some properties with critical percolation interfaces, but locally, they rather behave like off-critical percolation interfaces.
\end{abstract}

\section{Introduction}

The phase transition of site percolation on the triangular lattice is now mathematically well-understood. Smirnov's proof of conformal invariance in the scaling limit \cite{Sm1} has enabled to prove the convergence of critical interfaces to Schramm's SLE process with parameter $6$. SLE-based computations by Lawler, Schramm and Werner (see e.g. \cite{LSW4}) completed the rigorous proof of the existence and values of the so-called ``arm exponents'' for critical percolation \cite{SmW}. Combining these results with Kesten's scaling relations \cite{Ke4}, one gets a rather precise description of percolation near criticality in two dimensions \cite{SmW, N2}. Further ongoing work in this direction by Garban, Pete and Schramm \cite {GPS2}, where they construct the scaling limit of near-critical percolation, completes the picture.

There are two ways to describe the large-scale behavior of such lattice models. The first one is to identify the scaling limit of the geometric objects (the interfaces, the clusters). In the case of percolation, this is the convergence of interfaces or of collections of loops \cite {Sm1, Sm2, CN1, CN2} to the corresponding SLE-related conformally invariant objects. Note that usually, compactness arguments (for instance in the setup of Aizenman-Burchard \cite {AB}) provide existence of subsequential limits, but that (so far), the uniqueness of this limit has to rely on additional information such as conformal invariance. In recent ongoing work, Schramm and Smirnov \cite {ScSm} have proposed an elegant setup to describe the scaling limit as a random object in a nice space, the space of ``crossed quads''.

The other way to describe the large-scale behavior of such models is more implicit. One understands the scaling limits of certain probabilities (or their decay rate). This leads for instance to the identification of the so-called critical exponents for the models that had been predicted by Conformal Field Theory, and that are often directly related to the fractal dimensions of random sets defined by the previous setup.  An example of such a result is the power law for the characteristic length (measuring the mean size of a finite cluster \cite{Ke4})
$$\xi(p) = |p-1/2|^{-\nu+o(1)}$$
as $p \to 1/2$, with $\nu = 4/3$.

In \cite{N1}, we studied the gradient percolation model, an inhomogeneous percolation process where the density of black (occupied) sites depends on the location in space, that was introduced in \cite{Sa1} as an approximation of more complex systems where inhomogeneity plays a central role. Considering a strip of length $\ell_N$ and of finite width $2N$, with a parameter $p(y)$ decreasing linearly along the vertical axis from $1$ to $0$ (see Figure \ref{strip}), we showed that if the length of the strip satisfies $\ell_N \gg N^{\nu/(1+\nu)}$ ($=N^{4/7}$), then with high probability, there exists a unique ``front'', an interface between the cluster of black sites connected to the bottom of the strip and the cluster of white sites connected to the top, and the vertical fluctuations of this front are of order $N^{\nu/(1+\nu)}$. We also proved that various other macroscopic quantities associated with it -- its discrete length for instance -- can be described via critical exponents related to the exponents of standard percolation: it inherits many global properties of the critical percolation interfaces. These results are all results about ``exponents'' that give information about the length, the width etc. of gradient percolation interfaces.

\begin{figure}
\begin{center}
\includegraphics[width=10cm]{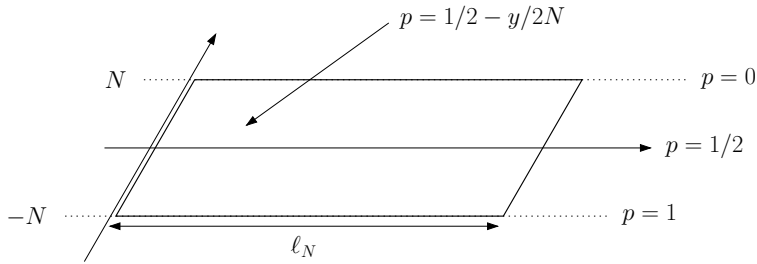}
\caption{\label{strip}Gradient percolation in a strip.}
\end{center}
\end{figure}

Note that in this case, the ``scaling limit'' of the interface that one sees by directly ``looking at the picture'' is a straight line, even though the number of points on the interface does not grow linearly.  Our results however strongly suggested that in order to construct non-trivial scaling limits for the front, one should scale it by a factor of order $N^{4/7}$ -- instead of $N$, for which one just gets a straight line in the limit (see Figure \ref{front_pic}). However, the existence of such scaling limits was not established, mainly due to possible logarithmic corrections in all estimates coming from SLE computations.

\begin{figure}
\begin{center}
\includegraphics[width=12cm]{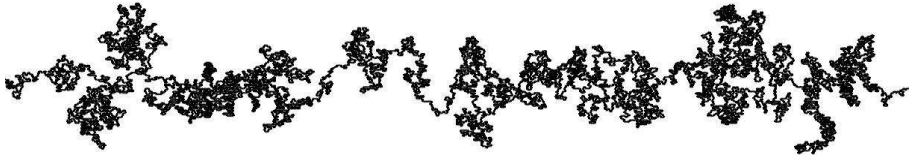}
\caption{\label{front_pic}The gradient percolation interface on a scale (roughly) $N^{4/7}$ -- on a scale of order $N$, one just sees a straight line.}
\end{center}
\end{figure}

The goal of the present paper is to study the scaling limits of the gradient percolation model interfaces, by combining ideas and results of \cite{N1} and \cite{NW}. We identify a -- both horizontal and vertical -- \emph{characteristic length}  $\sigma_N$ for this model, that turns out to be the right way to scale the front. We use this length to tighten results on gradient percolation obtained in \cite{N1}, which allows us to apply tightness arguments due to Aizenman and Burchard to construct non-trivial scaling limits.

We then study some properties of these scaling limits. Our construction shows that they are similar to interfaces in near-critical regime, studied in \cite{NW}. On the one hand, these interfaces share the same exponents as $SLE_6$, that describes limits of interfaces in critical regime. On the other hand, they are not SLE curves exactly: they are rather similar locally to interfaces in \emph{off-critical} regime.

Our results thus indicate that in a situation of inhomogeneity, where one sees self-critical interfaces that are localized where the density is close to the percolation threshold -- i.e. percolation phase transition appears spontaneously -- what we observe corresponds to the \emph{whole near-critical regime} rather than the critical regime exactly. In particular, the interfaces that arise are not \emph{stricto sensu} in the same class of universality as the critical interfaces -- even if global macroscopic quantities stay comparable to what they are at criticality.

\section{Characteristic length for gradient percolation}

\subsection{Percolation background}

This paper uses earlier results on near-critical percolation, primarily from \cite{Ke4, SmW}. In this section, we recall some results that we will use. All these results are stated and derived in \cite{N2}, and we follow the notations of this paper (see also \cite{W2} for a general account on two-dimensional critical percolation). In particular, we restrict ourselves to site percolation on the triangular lattice, with basis $(1,e^{i\pi/3})$, so that the parallelogram of corners $a_i + b_j e^{i\pi/3}$ ($i,j \in \{1,2\}$) is denoted by $[a_1,a_2]\times[b_1,b_2]$, and $\mathcal{C}_H([a_1,a_2]\times[b_1,b_2])$ refers to the existence of a horizontal black crossing in this parallelogram connecting its left side to its right side (we use $\mathcal{C}_V$ for vertical crossings). The notation $f \asymp g$ means that there exist two constants $C_1,C_2 > 0$ such that $C_1 g \leq f \leq C_2 g$, and $f \approx g$ means that ${\log f} / {\log g} \to 1$.

Recall the definition of the ``finite-size scaling'' characteristic length: for any $\epsilon \in (0,1/2)$,
\begin{equation}
L_{\epsilon}(p) = \min\{n \text{\: s.t. \:} \PP_p(\mathcal{C}_H([0,n] \times [0,n])) \leq \epsilon \}
\end{equation}
when $p<1/2$, and the same with white crossings when $p>1/2$ (so that $L_{\epsilon}(p) = L_{\epsilon}(1-p)$). This length measures the scale up to which percolation can be considered as ``almost'' critical, and it happened to be a key tool to study near-critical percolation -- see e.g. \cite{Ke4}.

We use the so-called arm events, more precisely the events
$$A_2 (n_1, n_2) = A_{2,BW} (n_1, n_2) = \{ \partial S_{n_1} \leadsto_{2,BW} \partial S_{n_2} \}$$
that there exist two arms, one black and one white, crossing the annulus $S_{n_1,n_2} = S_{n_2} \setminus \mathring{S}_{n_1}$ of radii $n_1$ and $n_2$ centered on the origin, and also
$$A_4 (n_1, n_2) = A_{4,BWBW} (n_1, n_2) = \{ \partial S_{n_1} \leadsto_{4,BWBW} \partial S_{n_2} \},$$
the similar event with four arms of alternating colors. We also introduce the notation, for $j=2$ or $4$,
$$\pi_j(n_1,n_2) = \PP_{1/2}(A_j (n_1, n_2))$$
(and when $n_1=0$, we choose not to mention it, so that $A_j(n)$ and $\pi_j(n)$ refer to $A_j(0,n)$ and $\pi_j(0,n)$ respectively).

\bigskip

The following properties hold for any fixed $\epsilon \in (0,1/2)$:

\begin{enumerate}
\item \label{RSW_it} The Russo-Seymour-Welsh estimates are valid below $L_{\epsilon}(p)$: for all $k \geq 1$, there exists $\delta_k = \delta_k(\epsilon) >0$ such that for all $p$, all $n \leq L_{\epsilon}(p)$, any product measure $\PPP$ between $\PP_p$ and $\PP_{1-p}$ (i.e. with associated parameters $(\hat{p}_v)$ so that for each site $v$, $\hat{p}_v$ is between $p$ and $1-p$),
\begin{equation}
\PPP(\mathcal{C}_H([0,k n] \times [0,n])) \geq \delta_k.
\end{equation}

\item \label{expdecay_it} The crossing probabilities decay exponentially fast with respect to $L_{\epsilon}(p)$: there exist constants $C_i=C_i(\epsilon)>0$ such that for all $p < 1/2$, all $n$,
\begin{equation} \label{expdecay_eq}
\PP_p(\mathcal{C}_H([0,n] \times [0,2n])) \leq C_1 e^{- C_2 n/L_{\epsilon}(p)}.
\end{equation}

\item \label{quasi_it} We have (this is known as the \emph{quasi-multiplicativity} property) for $j=2$ or $4$
\begin{equation} \label{quasi_prop}
\PPP (A_j (n_1 / 2)) \times \PPP (A_j (2 n_1, n_2)) \asymp \PPP (A_j (n_2))
\end{equation}
uniformly in $p$, $2 n_1 \leq n_2 \leq L_{\epsilon}(p)$ and $\PPP$ between $\PP_p$ and $\PP_{1-p}$, and also
\begin{equation} \label{near_critical_2arms}
\PPP (A_j (n_1,n_2)) \asymp \PP_{1 / 2} (A_j (n_1,n_2)).
\end{equation}

\item \label{arm_it} For $j=2$ or $4$, for any $\eta \in (0,1)$,
\begin{equation}
\PP_{1 / 2} (A_j (\eta n,n)) \to f_j (\eta)
\end{equation}
as $n \to \infty$, where $f_j (\eta ) = \eta^{\alpha_j + o(1)}$ as $\eta \to 0^+$, with $\alpha_2=1/4$ and $\alpha_4=5/4$. This implies (using Eq.(\ref{quasi_prop})) that
\begin{equation} \label{arm_exponent}
\pi_j(n) = \PP_{1 / 2} (A_j (n)) \approx n^{- \alpha_j}
\end{equation}
as $n \to \infty$.
\end{enumerate}

\subsection{Regularity and asymptotic behavior of $L_{\epsilon}$}

We will need precise estimates on the behavior of the characteristic length $L_{\epsilon}(p)$ as $p \to 1/2$. Let us first recall results that are derived in \cite{N2}. For any (fixed) $\epsilon \in (0,1/2)$, we have that
\begin{equation} \label{equiv}
|p-1/2| L_{\epsilon}(p)^2 \pi_4(L_{\epsilon}(p)) \asymp 1
\end{equation}
as $p \to 1/2$.

\begin{itemize}
\item Using Eq.(\ref{arm_exponent}) with $j=4$, this relation implies that for any $\epsilon \in (0,1/2)$,
\begin{equation} \label{exponent}
L_{\epsilon}(p) \approx |p - 1/2|^{-\nu}
\end{equation}
as $p \to 1/2$, with $\nu = 4/3$.

\item By combining this relation with an a-priori bound for $\pi_4$, namely that
\begin{equation} \label{4arms_apriori}
\pi_4(n_1,n_2) \geq c (n_1/n_2)^{2-\alpha}
\end{equation}
for any $n_1 < n_2$ ($c, \alpha > 0$ being universal constants), we also obtain that for any two $\epsilon,\epsilon' \in (0,1/2)$,
\begin{equation} \label{nonrelevance}
L_{\epsilon}(p) \asymp L_{\epsilon'}(p).
\end{equation}
\end{itemize}

In the following, we will need to compare $L_{\epsilon}(1/2+\delta)$ and $L_{\epsilon}(1/2+\delta')$ for two small values $\delta, \delta'>0$. We could write
$$\frac{L_{\epsilon}(1/2+\delta)}{L_{\epsilon}(1/2+\delta')} \approx \bigg( \frac{\delta}{\delta'}\bigg)^{-4/3},$$
but this logarithmic equivalence is not precise enough. We bypass this difficulty by deriving a weaker result that is sufficient for our purpose.

\begin{lemma} \label{regularity_lem}
There exist universal constants $C_1, C_2, \alpha_1, \alpha_2 >0$ such that
\begin{equation}
C_1 \bigg( \frac{\delta}{\delta'}\bigg)^{-\alpha_1} \leq \frac{L_{\epsilon}(1/2+\delta)}{L_{\epsilon}(1/2+\delta')} \leq C_2 \bigg( \frac{\delta}{\delta'}\bigg)^{-\alpha_2}
\end{equation}
for any two $0 < \delta < \delta' < 1/2$.
\end{lemma}

\begin{proof}
The proof is essentially the same as that of Eq.(\ref{nonrelevance}). We know from Eq.(\ref{equiv}) that
\begin{equation}
\delta \big(L_{\epsilon}(1/2+\delta)\big)^2 \pi_4(L_{\epsilon}(1/2+\delta)) \asymp 1 \asymp \delta' \big(L_{\epsilon}(1/2+\delta')\big)^2 \pi_4(L_{\epsilon}(1/2+\delta')),
\end{equation}
hence,
\begin{equation}
\frac{\big(L_{\epsilon}(1/2+\delta)\big)^2}{\big(L_{\epsilon}(1/2+\delta')\big)^2} \asymp \frac{\pi_4(L_{\epsilon}(1/2+\delta'))}{\pi_4(L_{\epsilon}(1/2+\delta))} \times \frac{\delta'}{\delta}.
\end{equation}
We deduce that
\begin{equation}
\frac{\big(L_{\epsilon}(1/2+\delta)\big)^2}{\big(L_{\epsilon}(1/2+\delta')\big)^2} \asymp \big( \pi_4(L_{\epsilon}(1/2+\delta'),L_{\epsilon}(1/2+\delta)) \big)^{-1} \times \frac{\delta'}{\delta}
\end{equation}
by quasi-multiplicativity (item \ref{quasi_it}. above, in the case of $j=4$ arms). The estimate for 4 arms Eq.(\ref{4arms_apriori}), and also the trivial bound $\pi_4(n_1,n_2) \leq 1$, now provide the desired conclusion.
\end{proof}

\subsection{Gradient percolation: setup}

Let us now define the gradient percolation model itself: we consider independent site percolation in the strip
\begin{equation}
\Sca^{\infty} = (-\infty,+\infty) \times [-N,N],
\end{equation}
with parameter
\begin{equation}
p(y) = \frac{1}{2}-\frac{y}{2N}.
\end{equation}
For our purpose, working with such a strip infinite in both directions will be more convenient: we get in this way a stationary process, and we avoid the boundary effects.

Note that with this setting, there is a.s. a unique interface between the infinite cluster of black sites connected to $(-\infty,+\infty) \times \{-N\}$ and the infinite cluster of white sites connected to $(-\infty,+\infty) \times \{N\}$: indeed, there exists a column consisting -- except for the top site -- only of black sites, and the interface can be explored starting from the top of this column. We refer to this (random) interface as the \emph{front}, and we denote it by $\FF_N$. We also introduce the sub-strips
\begin{equation}
\Sca^{t_1,t_2} = [t_1,t_2] \times [-N,N] \subseteq \Sca^{\infty}
\end{equation}
for $- \infty < t_1 < t_2 < +\infty$, and
\begin{equation}
[u_1,u_2] = (-\infty,+\infty) \times [u_1,u_2].
\end{equation}
We use in particular the notation $[\pm u] = [-u,u] = (-\infty,+\infty) \times [-u,u]$.

\subsection{Characteristic length for gradient percolation}

We now introduce a quantity $\sig$ that measures the vertical fluctuations of the front. As will become clear in the following, this quantity can be seen as a -- both horizontal and vertical -- characteristic length for the gradient percolation model.
\begin{definition}
For any $\epsilon  \in (0,1/2)$, any $N \geq 1$, we define
\begin{equation}
\sig = \sup \bigg\{ \sigma \text{ s.t. } L_{\epsilon}\big(p(\sigma)\big) = L_{\epsilon}\bigg(\frac{1}{2}-\frac{\sigma}{2N}\bigg) \geq \sigma\bigg\}.
\end{equation}
\end{definition}

Note that if we plug into this definition the value of the exponent $\nu$ associated with $L_{\epsilon}$ (i.e. Eq.(\ref{exponent})), we get that
\begin{equation}
\sig \approx N^{\nu/(1+\nu)} = N^{4/7}
\end{equation}
as $N \to \infty$. This implicit definition makes life easier compared to the closed value $N^{4/7}$ that we used in \cite{N1}, for which we had to take care of potential logarithmic corrections.

Eq.(\ref{nonrelevance}) implies that
\begin{equation}
\sig \asymp \sigma_N^{\epsilon'}
\end{equation}
for any two $\epsilon,\epsilon' \in (0,1/2)$: the scale $\sig$ is unique up to multiplicative constants. We thus fix some $\epsilon \in (0,1/2)$ (e.g. $\epsilon=1/4$) for the rest of the paper. We will see in the next section that $\sig$ is the right scale to consider.

Note that by definition of $\sig$, we have
\begin{equation} \label{comparable_p}
L_{\epsilon}(p(\sig)) \asymp \sig.
\end{equation}
On the one hand, we know that for any fixed $u \geq 1$, the Russo-Seymour-Welsh lower bounds hold in the strip $[\pm u \sig]$, uniformly for $N \geq 1$. On the other hand, the previous regularity lemma for $L_{\epsilon}$ (Lemma \ref{regularity_lem}) implies that for any $u \geq 1$,
\begin{equation}
L_{\epsilon}\big(p(u \sig)\big) = L_{\epsilon}\bigg(\frac{1}{2}-\frac{u \sig}{2N}\bigg) \leq C_1^{-1} u^{-\alpha_1} L_{\epsilon}\bigg(\frac{1}{2}-\frac{\sig}{2N}\bigg) = C_1^{-1} u^{-\alpha_1} L_{\epsilon}\big(p(\sig)\big),
\end{equation}
so that using Eq.(\ref{comparable_p}),
\begin{equation} \label{decay_p}
L_{\epsilon}\big(p(u \sig)\big) \leq c u^{-\alpha} \sig
\end{equation}
for some universal constants $c,\alpha>0$ (we would expect that $\alpha=4/3$, but one has to be careful with the logarithmic corrections). We will actually need Eq.(\ref{decay_p}) also for $u \in [1/2,1]$: this is again a consequence of Lemma \ref{regularity_lem} (just increase the constant $c$ if necessary).

\subsection{Macroscopic properties of the front}

The definition of $\sig$ enables us to tighten the estimates of \cite{N1}. The proofs are essentially the same, we recall them briefly since some non-trivial adaptations are needed. Our reasonings are based on the following two main observations:
\begin{itemize}
\item The front $\FF_N$ never goes far from the strip $[\pm \sig]$, due to the exponential decay property Eq.(\ref{expdecay_eq}).

\item The behavior of $\FF_N$ in any strip $[\pm u \sig]$ ($u \geq 1$) is roughly the same as that of critical percolation.
\end{itemize}

\subsubsection*{Localization}

\begin{proposition} \label{local}
There exist some universal constants $\alpha, C_1, C_2 > 0$ such that
\begin{equation} \label{loc1}
\PP(\FF_N \cap \Sca^{0, t \sig} \nsubseteq [\pm u \sig]) \leq C_1 \bigg( \frac{t}{u} + \frac{u}{t} \bigg) e^{- C_2 u^{\alpha}}
\end{equation}
for all $t, u \geq 1$.
\end{proposition}

\begin{proof}
Assume that $\FF_N \cap \Sca^{0, t \sig} \nsubseteq [\pm u \sig]$. If for instance $\FF_N \cap \Sca^{0, t \sig}$ exits the strip $[\pm u \sig]$ from above, then either it stays above the line $y=\frac{u}{2} \sig$ and crosses horizontally the strip $[0,t \sig] \times [\frac{u}{2} \sig, N]$, either it crosses vertically the strip $[0,t \sig] \times [\frac{u}{2} \sig,u \sig]$.

Consider the first case: the event $\mathcal{C}_H([0,t \sig] \times [\frac{u}{2} \sig, N])$ occurs. If $t \geq u$, there is a horizontal crossing in one of the parallelograms $[0,u \sig] \times [(\frac{u}{2} + ku) \sig, (\frac{u}{2} + (k+2)u) \sig]$ ($k=0,1, \ldots$), or a vertical crossing in one of the parallelograms $[0,2 u \sig] \times [(\frac{u}{2} + k'u) \sig, (\frac{u}{2} + (k'+1)u) \sig]$ ($k'=0,1, \ldots$). Using the exponential decay property Eq.(\ref{expdecay_eq}), we get
\begin{align*}
\PP\bigg(\mathcal{C}_H\Big([0,u \sig] \times \Big[\Big(\frac{u}{2} + ku\Big) \sig, \Big(\frac{u}{2} + (k+2)u\Big) \sig\Big]\Big)\bigg) & \\
& \hspace{-3cm} \leq C_1 e^{- C_2 u \sig/L_{\epsilon}(p((\frac{u}{2} + ku) \sig))} \\
& \hspace{-3cm} \leq C_1 e^{- C_3 (k+\frac{1}{2})^{\alpha} u^{\alpha}}
\end{align*}
since $L_{\epsilon}(p((\frac{u}{2} + ku) \sig)) \leq c ((k+\frac{1}{2}) u)^{-\alpha} \sig$ (by Eq.(\ref{decay_p})). The same bound holds for $\PP(\mathcal{C}_V([0,2 u \sig] \times [(\frac{u}{2} + k'u) \sig, (\frac{u}{2} + (k'+1)u) \sig]))$, and by summing over $k,k' \geq 0$, we get that the considered event has a probability at most
\begin{align*}
2 \sum_{k''=1}^{\infty}  C_1 e^{- C_3 2^{-\alpha} k''^{\alpha} u^{\alpha}} & \leq 2 C_1 e^{- C_4 u^{\alpha}} \sum_{k''=1}^{\infty}  e^{- C_4 (k''^{\alpha} - 1)}\\
& \leq C_5 \frac{t}{u} e^{- C_4 u^{\alpha}},
\end{align*}
by using that $u^{\alpha} \geq 1$ and $t/u \geq 1$.

If $t < u$, we use parallelograms of size $t$ instead, and the considered event has a probability at most:
\begin{align*}
2 \sum_{k=0}^{\infty} C_1 e^{- C_3 (\frac{u}{2}+k t)^{\alpha}} & \leq C_6 \frac{u}{t} \sum_{l=1}^{\infty} e^{- C_7 l^{\alpha} u^{\alpha}}\\
& \leq C_8 \frac{u}{t} e^{- C_7 u^{\alpha}},
\end{align*}
since for each $l' \geq 0$, there are of order $u/t$ values of $k$ for which $(l'+\frac{1}{2}) u \leq \frac{u}{2}+k t < (l'+\frac{3}{2}) u$.

Consider now the second case: the event $\mathcal{C}_V([0,t \sig] \times [\frac{u}{2} \sig,u \sig])$ occurs. If $t \geq u$, there is a vertical crossing in one of the parallelograms $[k u \sig, (k+1) u \sig] \times [\frac{u}{2} \sig,u \sig]$ ($0 \leq k u \leq t$), or a horizontal crossing in one of the parallelograms $[k' \frac{u}{2} \sig, (k'+1) \frac{u}{2} \sig] \times [\frac{u}{2} \sig, 3 \frac{u}{2} \sig]$ ($0 \leq k'\frac{u}{2} \leq t$). There are of order $t/u$ such parallelograms, and for each of them there is a crossing with probability at most
$$C_1 e^{- C_2 \frac{u}{2} \sig /L_{\epsilon}(p(\frac{u}{2} \sig))} \leq C_1 e^{-C_3 u^{\alpha}}.$$
The considered event has thus a probability at most
$$C_4 \frac{t}{u} e^{-C_3 u^{\alpha}}.$$
If $t<u$, we use the parallelogram $[0, t \sig] \times [\frac{u}{2} \sig,u \sig]$: it is crossed vertically with probability at most
$$C_1 \frac{u}{t} e^{-C_3 u^{\alpha}},$$
since $u/t \geq 1$.
\end{proof}

\begin{remark}
In the other direction, let us fix some $u\geq 1$. For $t \geq u$, we can consider the independent parallelograms $[k u \sig,(k+1) u \sig] \times [-u \sig,u \sig]$ ($0 \leq k \leq \frac{t}{u}-1$): in each of them, a vertical black crossing occurs with probability at least $\delta_2 = \delta_2(u) > 0$, so that
\begin{equation} \label{loc2}
\PP(\FF_N \cap \Sca^{0, t \sig} \subseteq [\pm u \sig]) \leq (1 - \delta_2)^{t/u-1} \leq C_1 e^{- C_2 t},
\end{equation}
for some constants $C_1,C_2 >0$ (depending on $u$).
\end{remark}

\subsubsection*{Uniqueness}

\begin{proposition} \label{uniq}
There exist universal constants $C_1, C_2 >0$ such that
\begin{equation}
\PP(\text{ there is a unique crossing in $\Sca^{0,t \sig}$ }) \geq 1 - C_1 e^{-C_2 t}.
\end{equation}
\end{proposition}

\begin{proof}

We need the following property: we state it as a separate lemma since we will use it again later.
\begin{lemma} \label{vertical_crossing}
For each fixed $v_1 >0$ and $v_2$, there exists a constant $\eta(v_1,v_2) > 0$ independent of $N$ such that
\begin{equation}
\PP\big(\mathcal{C}_V([0,v_1\sig] \times [-N,v_2 \sig])\big) \geq \eta(v_1,v_2).
\end{equation}
\end{lemma}

\begin{proof}
We construct a vertical crossing in $[0,v_1\sig] \times [-N,v_2 \sig]$ by combining vertical crossings in the parallelograms $[0,v_1\sig] \times [(v_2-(k+2)v_1)\sig,(v_2-k v_1)\sig]$ ($k=0,1, \ldots$) and horizontal crossings in the parallelograms $[0,v_1\sig] \times [(v_2-(k'+1)v_1)\sig,(v_2-k'v_1)\sig]$ ($k'=1,2, \ldots$). It is easy to check that with probability at least
\begin{equation}
C_1(v_1,v_2) \prod_{k=k_0}^{\infty} (1 - C_2 e^{- C_3 k^{\alpha}}) = \eta(v_1,v_2) > 0,
\end{equation}
all these crossings exist, using (as for localization) the exponential decay property Eq.(\ref{expdecay_eq}) and Eq.(\ref{decay_p}).
\end{proof}

Consider now the strip $\Sca^{0, 3 \sig}$. Let us condition on the upper-most crossing $r_N$ in this strip. We know from Lemma \ref{vertical_crossing} that
\begin{equation}
\PP\big(\mathcal{C}^*_V([\sig,2\sig] \times [0,N])\big) \geq \eta = \eta(1,0) > 0,
\end{equation}
a lowest point on $r_N$ in $\Sca^{\sig, 2 \sig}$ thus lies below the $x$-axis with probability at least $\eta$: let us assume that this is the case. The construction of Figure \ref{uniqueness} then shows that with probability at least $\delta_4^4 \eta'$, with $\eta' = \eta(1/10,0)$, $r_N$ is connected to the bottom of the strip $\Sca^{0, 3 \sig}$, by a path staying in this strip.

\begin{figure}
\begin{center}
\includegraphics[width=10cm]{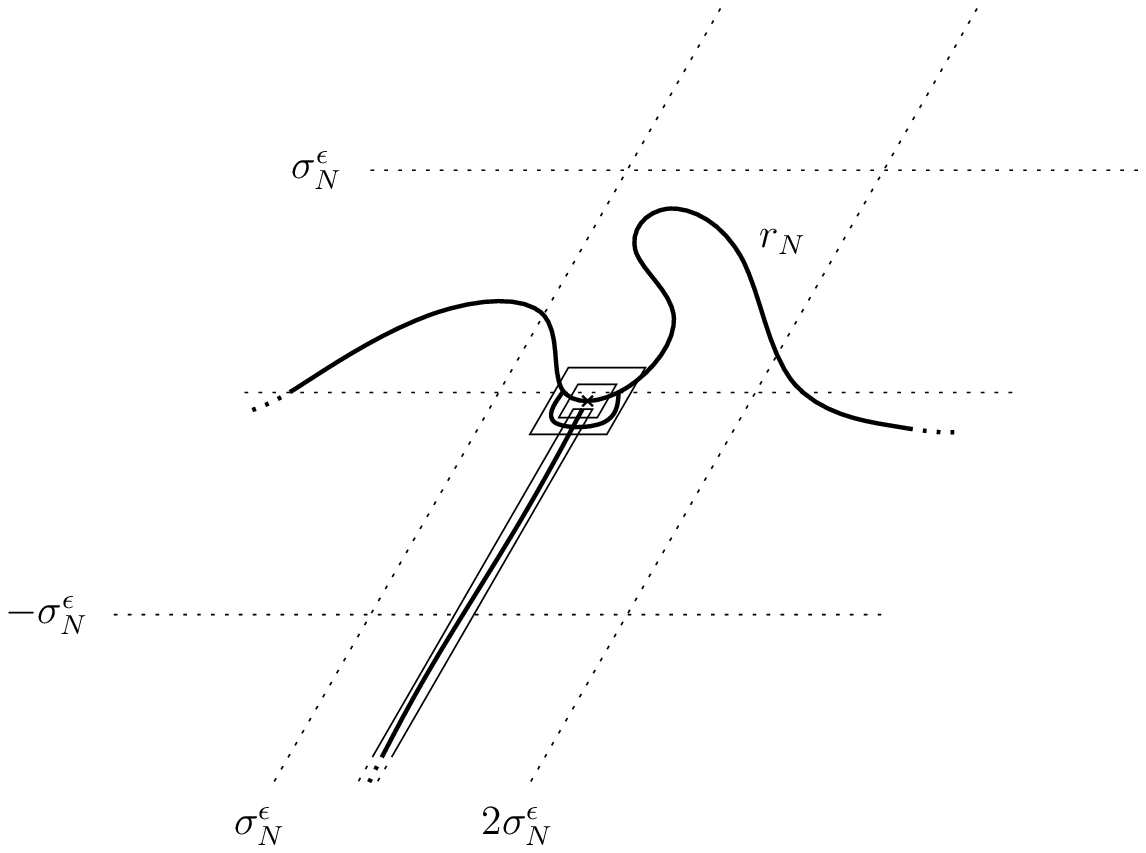}
\caption{\label{uniqueness}With positive probability, the front is connected to the bottom of the strip $\Sca^{0, 3 \sig}$.}
\end{center}
\end{figure}

Now, consider $c t$ disjoint sub-strips $\Sca^{(i)}$ of length $3 \sig$ in the strip $\Sca^{0, t \sig}$, for $c>0$ a small constant: for each $i$, $\FF_N$ is connected to the bottom of $\Sca^{(i)}$ (by a path staying in $\Sca^{(i)}$) with probability at least $\delta' = \delta_4^4 \eta \eta'$, so that $\FF_N$ is connected to the bottom of $\Sca^{0, t \sig}$ with probability at least
$$1 - (1-\delta')^{c t} \geq 1 - C_1 e^{-C_2 t}.$$
\end{proof}

\begin{remark}
Note that this uniqueness property implies in particular that the front does not ``bounce'' backwards too much: we would otherwise observe multiple interfaces in some sub-strips (see Figure \ref{bounce}), which has a very small probability to happen. We use this remark in Section \ref{scalinglimit}, when constructing scaling limits. We have for instance:
\begin{equation}
\PP\big( \text{ $\FF_N$ bounces backwards by a distance $\geq u \sig$ in $\Sca^{0,t\sig}$ } \big) \leq C_1 \frac{t}{u} e^{-C_2 u}.
\end{equation}

\begin{figure}
\begin{center}
\includegraphics[width=10cm]{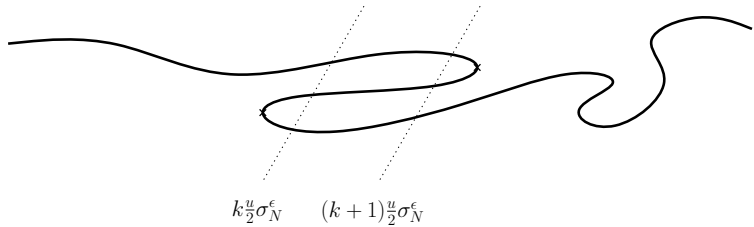}
\caption{\label{bounce}If the front bounced backwards by a distance larger than $u \sig$, it would create several interfaces in one of the sub-strips $\Sca^{k \frac{u}{2} \sig,(k+1) \frac{u}{2} \sig}$.}
\end{center}
\end{figure}
\end{remark}

\subsubsection*{Length}

\begin{proposition} \label{length_front}
The following estimate on the discrete length of $\FF_N$ (its number of edges) holds:
\begin{equation}
\EE\big[|\FF_N \cap \Sca^{0,t\sig}|\big] \asymp t (\sig)^2 \pi_2(\sig)
\end{equation}
uniformly for $t \geq 1$ and $N \geq 1$.
\end{proposition}

\begin{proof}
It comes from Eq.(\ref{near_critical_2arms}) and Lemma \ref{vertical_crossing} (actually it also uses \emph{separation of arms}, see \cite{N2}) that
$$\PP(e \in \FF_N) \asymp \pi_2(\sig)$$
uniformly for $e \in [\pm \sig]$, which provides the lower bound.

\begin{figure}
\begin{center}
\includegraphics[width=8cm]{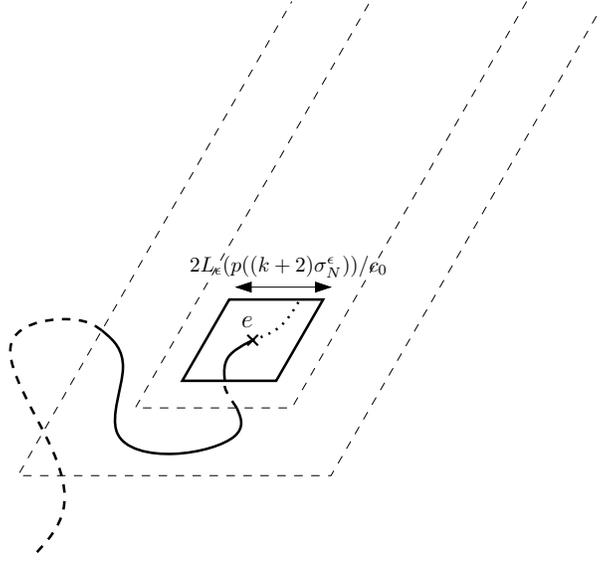}
\caption{\label{twoarms_fig}If $e \in [k \sig,(k+1)\sig]$ is on $\FF_N$, then there are two arms from $e$ going to distance $L_{\epsilon}(p((k+2)\sig))/c_0$, and a black crossing in the U-shaped region of ``width'' $\sig$.}
\end{center}
\end{figure}

On the other hand, the construction of Figure \ref{twoarms_fig} implies that if $e \in [k \sig,(k+1)\sig]$ for some $k \geq 0$,
\begin{equation} \label{vertical_decay}
\PP(e \in \FF_N) \leq C_1 e^{-C_2 k^{\alpha}} \pi_2\big(L_{\epsilon}(p((k+2)\sig))/c_0\big),
\end{equation}
by using Eq.(\ref{near_critical_2arms}) (for the probability of having two arms) and by combining (for the crossing in the U-shaped region) Eqs.(\ref{expdecay_eq}) and (\ref{decay_p}) as for localization, where $c_0$ is chosen large enough so that $L_{\epsilon}(p(\sig))/c_0 \leq \sig$ (this ensures that the box around $e$ is included in $[(k-1) \sig,(k+2) \sig]$). We have
\begin{align*}
\pi_2(L_{\epsilon}(p((k+2)\sig))/c_0) & \leq C_3 \pi_2(L_{\epsilon}(p((k+2)\sig)))\\
& \leq C_4 \pi_2\big(L_{\epsilon}(p((k+2)\sig)),L_{\epsilon}(p(\sig))\big)^{-1} \pi_2(L_{\epsilon}(p(\sig)))\\
& \leq C_5 \bigg( \frac{L_{\epsilon}(p((k+2)\sig))}{L_{\epsilon}(p(\sig))} \bigg)^{-\tilde{\alpha}} \pi_2(\sig)\\
& \leq C_6 \big( (k+2)^{-\alpha'} \big)^{- \tilde{\alpha}} \pi_2(\sig),
\end{align*}
by quasi-multiplicativity (item \ref{quasi_it}. above) and Lemma \ref{regularity_lem} (we also used an a-priori bound on $\pi_2$, that $\pi_2(n_1,n_2) \geq \tilde{c} (n_1/n_2)^{\tilde{\alpha}}$). This provides the upper bound by summing over $k \geq 0$:
\begin{align*}
\EE\big[|\FF_N \cap \Sca^{0,t\sig}|\big] & \leq 2 \sum_{k=0}^{\infty} t (\sig)^2 C_1 e^{-C_2 k^{\alpha}} \pi_2\big(L_{\epsilon}(p((k+2)\sig))/c_0\big)\\
& \leq C_7 t (\sig)^2 \pi_2(\sig) \bigg(\sum_{k=0}^{\infty} (k+2)^{\alpha' \tilde{\alpha}} e^{-C_2 k^{\alpha}} \bigg)\\
& \leq C_8 t (\sig)^2 \pi_2(\sig).
\end{align*}
\end{proof}


\section{Scaling limits}

\subsection{Existence} \label{scalinglimit}

Eqs.(\ref{loc1}) and (\ref{loc2}) show that $\sig$ is the right way (up to multiplicative constants) to scale $\FF_N$ in order to get non-trivial scaling limits. We thus consider the interface $\ff_N$ obtained by scaling $\FF_N$ in both directions by $\sig$ (note that $\ff_N$ depends on $\epsilon$). The now-classic tightness arguments due to Aizenman and Burchard \cite{AB} will allow us to construct scaling limits, and Eq. (\ref{loc2}) will then ensure that the scaling limits so obtained are non-trivial.

Let us make the setting a bit more precise. We work with the following \emph{space of interfaces} $\mathcal{S}$: we consider the set of continuous functions $\gamma : (-\infty,+\infty) \longrightarrow \mathbb{R}^2$ such that the $x$-coordinate of $\gamma(t)$ tends to $+\infty$ as $t \to +\infty$, and to $-\infty$ as $t \to -\infty$, where we identify two curves $\gamma$ and $\tilde{\gamma}$ if they are the same up to reparametrization, i.e. if there exists an increasing bijection $\phi : (-\infty,+\infty) \longrightarrow (-\infty,+\infty)$ such that $\gamma = \tilde{\gamma} \circ \phi$.

Recall that we usually endow the space of curves defined on a compact, say $[0,1]$, with the uniform distance up to reparametrization
\begin{equation}
d(\gamma_1,\gamma_2) = \inf_{\phi} \sup_{t \in [0,1]} |\gamma_1(t) - \gamma_2(\phi(t))|,
\end{equation}
where the infimum is taken over the set of increasing bijections $\phi: [0,1] \longrightarrow [0,1]$. The same distance over $(-\infty,+\infty)$ would be too strong for our purpose, we rather use the particular structure of $\mathcal{S}$ to define a notion of uniform convergence on every compact subinterval of $(-\infty,+\infty)$.

For a curve $\gamma \in \mathcal{S}$, we consider its piece $\gamma^{(n)}$ between $\underline{t}^{(n)}$ the first time it reaches $x=-n$ and $\overline{t}^{(n)}$ the last time it hits $x=n$. We can choose the parametrization such that $\underline{t}^{(n)} = -n$ and $\overline{t}^{(n)}=n$, so that $\gamma^{(n)}$ is parametrized by $[-n,n]$. We then define
\begin{equation}
d^{(n)}(\gamma_1,\gamma_2) = \inf_{\phi} \sup_{t \in [-n,n]} |\gamma_1^{(n)}(t) - \gamma_2^{(n)}(\phi(t))|,
\end{equation}
for $\phi : [-n,n] \longrightarrow [-n,n]$, and finally the product distance
\begin{equation}
d(\gamma_1,\gamma_2) = \sum_{n=1}^{+\infty} \frac{1}{2^n}\big(d^{(n)}(\gamma_1,\gamma_2) \wedge 1\big).
\end{equation}

One can check (see \cite{B_book} for instance) that a sequence $(\gamma_k)$ of interfaces converges in distribution toward $\gamma$ in $(\mathcal{S},d)$ \emph{iff} each $\gamma_k^{(n)}$ converges toward $\gamma^{(n)}$ in $(\mathcal{S}^{(n)},d^{(n)})$, with obvious notation for $\mathcal{S}^{(n)}$. Note also that in our setting, tightness and relative compactness are equivalent, by Prohorov's theorem.

The scaled fronts $\ff_N$ are elements of $\mathcal{S}$, and we are in a position to use the arguments of Aizenman and Burchard {\cite{AB}} in each $\mathcal{S}^{(n)}$.

\begin{proposition}
Denoting by $P_N$ the law of $\ff_N$, the sequence $(P_N)_{N \geq 1}$ is relatively compact in the set of probability measures on $(\mathcal{S},d)$.
\end{proposition}

\begin{proof}
We show the existence of scaling limits in each $\mathcal{S}^{(n)}$. Let us fix an integer $n \geq 1$ and $\epsilon>0$. First, there exist $t, u \geq 1$ such that for all $N$,
\begin{equation}
\PP(\ff_N^{(n)} \subseteq [-(n+t) ,(n+t) ] \times [-u ,u ]) \geq 1- \epsilon/2.
\end{equation}
Indeed, it comes from uniqueness in the strips $\Sca^{-(n+t) \sig,-n \sig}$ and $\Sca^{n \sig, (n+t) \sig}$ (Proposition \ref{uniq}), and then localization for $\FF_N \cap \Sca^{-(n+t) \sig,(n+t) \sig}$ (Proposition \ref{local}).

We can now use Theorem 1.2 of \cite{AB}. Indeed, since the Russo-Seymour-Welsh estimates hold in the strip $[\pm u \sig]$ uniformly for all $N$, we have for any annulus ${\mathcal A} (z; r, R) =\{\tilde{z} \text{\: s.t. \:} r < |\tilde{z} - z| < R\}$,
$$\PP ( \mathcal{A} (z; r, R)  \text{ is crossed by } \ff_N^{(n)}) \leq  c (r / R)^{\alpha}$$
for two universal constants $\alpha, c >0$, and the BK inequality implies that
$$\PP ( \mathcal{A}  (z; r, R) \text{ is crossed $k$ times by } \ff_N^{(n)}) \leq c^{k/2} (r / R)^{\alpha k/2}$$ which is exactly the hypothesis (H1) of {\cite{AB}} (uniform power bounds on the probability of multiple crossings in annuli). Hence, the sequence $(P_N^{(n)})_{N \geq 1}$, consisting of the laws of the $\ff_N^{(n)}$, is tight. This proves that the sequence $(P_N)_{N \geq 1}$ is relatively compact (using a diagonal argument).
\end{proof}

\begin{remark}
Up to now, for the definition of $\sig$ as well as for the subsequent reasonings, we have made no mention of the existence of the exponent $\nu$. The previous results are thus valid on other regular lattices, for which conformal invariance has not been established but RSW is known to hold, like the square lattice $\mathbb{Z}^2$. For this lattice, we know that
$$L_{\epsilon}(p) \leq C |p-p_c|^{-A}$$
for some $C,A>0$ (see \cite{N2}), which implies that for  some $\alpha>0$,
$$\sig \leq N^{1-\alpha}.$$
Hence, our proofs are still valid in this case: the front is unique, it converges toward the line $p=p_c$, and scaling it by $\sig$ produces non-trivial scaling limits.
\end{remark}

\subsection{Properties}

In this final section, we briefly discuss some properties of the potential scaling limits of the front, coming from the way they were constructed. Their behavior is comparable to that of near-critical interfaces, studied in \cite{NW}: indeed, we have seen that the front remains (with high probability) in a region where $L_{\epsilon}(p(y)) = O(y)$. We informally describe how to adapt the results and ideas developed in \cite{NW} to the present inhomogeneous setting -- since no real modification is needed compared to the near-critical case, we do not repeat the arguments in detail.

\begin{itemize}

\item{Similarities with critical percolation}

On the one hand, these scaling limits share some properties with scaling limits of critical percolation interfaces, i.e. $SLE_6$. On the discrete level, the probabilities of arm events remain the same up to multiplicative constants (we have stated it for $2$ or $4$ arms only, but this is true for any number of arms), and we have seen that this implies for instance that the discrete length remains comparable to what it is in critical regime (this is Proposition \ref{length_front}).

In the scaling limit, the Hausdorff dimension of any limit must be the same as that of critical percolation interfaces: if $\ff$ is a weak limit along some sub-sequence $(\ff_{N_k})$ of the sequence of interfaces, then the Haussdorff dimension of $\ff$ is almost surely $7/4$.

%

%
%

\item{Local asymmetry}

On the other hand, when the front is macroscopically far from $y=0$, i.e. at a distance of order $\sig$, it is by construction in a region where $L_{\epsilon}(p(y)) \asymp |y| \asymp \sig$. Gradient percolation thus provides a natural setting where the \emph{off-critical} regime, studied in \cite{NW}, arises.

On the discrete level, one can notice that a hexagon $h$ is on the front \emph{iff} there exist two arms from $h$ to the bottom and top boundaries of the strip (respectively black and white), no matter the state of $h$. Hence, denoting by $\partial \FF_N$ (resp. $\partial^+ \FF_N$, $\partial^- \FF_N$) the set of hexagons adjacent to $\FF_N$ (resp. black hexagons, white hexagons),
\begin{align*}
\PP(h \in \partial^+ \FF_N) - \PP(h \in \partial^- \FF_N) & \\
& \hspace{-3cm} = \PP(h \in \partial \FF_N \text{ and $h$ is black}) - \PP(h \in \partial \FF_N \text{ and $h$ is white})\\
& \hspace{-3cm} = \PP(\text{ two arms from $h$ }) \times (2 p(y_h) -1).
\end{align*}
Consequently, in any sub-strip $s=[0,t \sig] \times [u_1 \sig,u_2 \sig]$ ($u_1 < u_2 < 0$), the following asymmetry property holds:
\begin{equation}
\mathbb{E}\big[|\partial^+ \FF_N \cap s|\big] - \mathbb{E}\big[|\partial^- \FF_N \cap s|\big] \asymp \frac{\sig}{N} \times \big( t (\sig)^2 \pi_2(\sig) \big) \approx N^{4/7},
\end{equation}
which is much larger than the statistical deviation
\begin{equation}
\sqrt{\mathbb{E}\big[|\partial \FF_N \cap s|\big]} \asymp \sqrt{t (\sig)^2 \pi_2(\sig)} \approx \sqrt{N}.
\end{equation}
Hence, there is a noticeable excess of black hexagons over white hexagons on the interface, which thus turns more in one direction.

If $\ff$ is the weak limit along some sub-sequence $(\ff_{N_k})$ of the sequence of interfaces, the same type of \emph{local} asymmetry as for off-critical interfaces in \cite{NW} should hold for any (small) portion of $\ff$ located in a strip of the form $[u_1,u_2]$ ($u_1 < u_2 < 0$): the interface turns more in one direction, on every scale. In particular, its law is singular with respect to that of $SLE_6$ -- so that gradient percolation interfaces are not strictly speaking in the same class of universality as critical percolation interfaces.

\end{itemize}

\section*{Acknowledgments}

The author would like to thank W. Werner for many stimulating discussions, in particular for very useful comments and suggestions regarding this paper. This research was supported in part by the NSF under grant OISE-07-30136.

\end{document}